\documentclass{article}

\usepackage{amsmath, amsthm}%
\usepackage{amsfonts}%
\usepackage{amssymb}%
\usepackage{graphicx}
\usepackage{amscd, amstext, hyperref}
\usepackage{pictexwd}
\usepackage{dcpic}
\usepackage[mathscr]{eucal}

\newcommand{\scr}[1]{\ensuremath{\mathcal{#1}}}
\newcommand{\bb}[1]{\ensuremath{\mathbb{#1}}}
\newcommand{\into}{\hookrightarrow}

\theoremstyle{definition}
\newtheorem{theorem}{Theorem}[section]
\newtheorem{proposition}[theorem]{Proposition}
\newtheorem{remark}[theorem]{Remark}
\newtheorem{example}[theorem]{Example}
\newtheorem{definition}[theorem]{Definition}

\newtheorem{corollary}[theorem]{Corollary}
\newtheorem{lemma}[theorem]{Lemma}

\newcommand{\R}{\mathbb{R}}
\newcommand{\C}{\mathbb{C}}
\newcommand{\Z}{\mathbb{Z}}
\renewcommand{\P}{\mathbb{P}}

\newcommand{\derive}[2]{\frac{\partial^{#2} #1}{\partial t^{#2}}}

\begin{document}

\title{On a family of K3 surfaces with $\mathcal{S}_4$ symmetry}
\author{Dagan Karp, Jacob Lewis, Daniel Moore, Dmitri Skjorshammer, and Ursula Whitcher\thanks{We thank Andrey Novoseltsev for thoughtful discussion of computational techniques, and Charles Doran for inspirational conversations.  We gratefully acknowledge support by the National Science Foundation under the Grants
DMS-083996 and OISE-0965183.  Some of our computations use systems funded by National Science Foundation Grant No. DMS-0821725.
Any opinions, findings, and conclusions or
recommendations expressed in this material are those of the authors
and do not necessarily reflect the views of the National Science
Foundation.}}
\date{}
\maketitle

\abstract{The largest group which occurs as the rotational symmetries of a three-dimensional reflexive polytope is $S_4$.   There are three pairs of three-dimensional reflexive polytopes with this symmetry group, up to isomorphism.  We identify a natural one-parameter family of K3 surfaces corresponding to each of these pairs, show that $S_4$ acts symplectically on members of these families, and show that a general K3 surface in each family has Picard rank $19$.  The properties of two of these families have been analyzed in the literature using other methods.  We compute the Picard-Fuchs equation for the third Picard rank $19$ family by extending the Griffiths-Dwork technique for computing Picard-Fuchs equations to the case of semi-ample hypersurfaces in toric varieties.  The holomorphic solutions to our Picard-Fuchs equation exhibit modularity properties known as ``Mirror Moonshine''; we relate these properties to the geometric structure of our family.}

\section{Introduction}

Families of Calabi-Yau varieties with discrete symmetry groups provide a fertile source of examples and conjectures in geometry and theoretical physics.  Greene and Plesser's construction of the mirror to a family of Calabi-Yau threefolds relied on the construction of a special pencil of threefolds admitting a discrete group symmetry (see \cite{GP}).  More recent studies of Calabi-Yau threefolds with discrete symmetry groups include \cite{DGJ}, \cite{BvGK}, and \cite{LR}.

In the case of K3 surfaces, actions of a finite group of \emph{symplectic} automorphisms, which preserve the holomorphic two-form, are of particular interest.  Nikulin classified the finite abelian groups which can act symplectically on K3 surfaces in \cite{Nikulin}.
Mukai showed in \cite{Mukai} that any finite group $G$ with a symplectic action on a K3 surface is a subgroup of a member of a list of eleven groups, and gave an example of a symplectic action of each of these maximal groups.  Xiao and Kond{\=o} gave alternate proofs of the classification in \cite{Xiao} and \cite{Kondo}, respectively; \cite[Table 2]{Xiao} includes a complete list of finite groups which admit symplectic group actions on K3 surfaces.  

If a K3 surface $X$ admits a symplectic action by a group $G$, then the Picard group of $X$ must contain a primitive definite sublattice $S_G$; in \cite{CommAlg}, the third author gives a procedure for computing the lattice invariants of $S_G$ for any of the groups in \cite[Table 2]{Xiao}.  The relationship between a symplectic action and the Picard group has been worked out in detail for particular finite groups: cf. \cite{OZ},  \cite{Gar08}, \cite{Gar09}, \cite{Garbagnati}, \cite{GarSarti} and \cite{Hashimoto}.  Thus, symplectic group actions may be used to identify K3 surfaces with high Picard rank.  

Families of K3 surfaces with Picard rank $19$ admit a particularly nice construction of the mirror map, which relates the moduli of a family of K3 surfaces to the moduli of the mirror family (see \cite{Dolgachev}).  One may study the mirror map using Picard-Fuchs differential equations.  In \cite{Doran}, Doran uses Picard-Fuchs differential equations to determine when a mirror map for a Picard rank $19$ family of K3 surfaces is an automorphic function.  The dissertation \cite{Smith} uses symplectic group actions to produce pencils of K3 surfaces with Picard rank $19$ in projective space $\mathbb{P}^3$ and the weighted projective space $\mathbb{P}(1,1,1,3)$, and computes the associated Picard-Fuchs equations.  

In this paper, we study special one-parameter families of K3 hypersurfaces in toric varieties obtained from three-dimensional reflexive polytopes.  The key idea is to use symmetries of the polytopes to identify a group action on a family of hypersurfaces.  The largest group which occurs as the rotational symmetries of a three-dimensional reflexive polytope is $S_4$.  Up to automorphism, there are three pairs of three-dimensional reflexive polytopes with this symmetry group.  We identify a natural one-parameter family of K3 surfaces corresponding to each of these pairs, show that $S_4$ acts symplectically on members of these families, and show that a general K3 surface in each family has Picard rank $19$.  

The Picard-Fuchs equations for two of our families were analyzed in \cite{PS}, \cite{Verrill}, and \cite{HLOY}.  We compute the Picard-Fuchs equation for the third Picard rank $19$ family, using coordinates which arise naturally from the reflexive polytope.  In order to do so, we extend the Griffiths-Dwork algorithm to the case of semi-ample hypersurfaces in toric varieties.  Our method relies on the theory of residue maps for hypersurfaces in toric varieties developed in \cite{BC} and extended in \cite{Mavlyutov}.  


\section{Toric varieties and semiample hypersurfaces}

\subsection{Toric varieties and reflexive polytopes}

We begin by recalling some standard constructions involving toric varieties.  Let $N$ be a lattice isomorphic to $\mathbb{Z}^n$.  The dual lattice $M$ of $N$ is given by $\mathrm{Hom}(N, \mathbb{Z})$; it is also isomorphic to $\mathbb{Z}^n$.  We write the pairing of $v \in N$ and $w \in M$ as $\langle v , w \rangle$.  A \emph{cone} in $N$ is a subset of the real vector space $N_\mathbb{R} = N \otimes \mathbb{R}$ generated by nonnegative $\mathbb{R}$-linear combinations of a set of vectors $\{v_1, \dots , v_m\} \subset N$.  We assume that cones are strongly convex, that is, they contain no line through the origin.  Note that each face of a cone is a cone.  

A \emph{fan} $\Sigma$ consists of a finite collection of cones such that each face of a cone in the fan is also in the fan, and any pair of cones in the fan intersects in a common face.  We say $\Sigma$ is \emph{simplicial} if the generators of each cone in $\Sigma$ are linearly independent over $\mathbb{R}$.  If every element of $N_\R$ belongs to some cone in $\Sigma$, we say $\Sigma$ is \emph{complete}.  In the following, we shall restrict our attention to complete fans.

A fan $\Sigma$ defines a toric variety $V_\Sigma$.  We may describe $V_\Sigma$ using homogeneous coordinates, in a process analogous to the construction of $\mathbb{P}^n$ as a quotient space of $(\mathbb{C}^*)^n$.  We follow the exposition in \cite{MirrorSymmetry}.  Let $\Sigma(1) = \{\rho_1, \dots, \rho_q\}$ be the set of one-dimensional cones of $\Sigma$.  For each $\rho_j \in \Sigma(1)$, let $v_j$ be the unique generator of the semigroup $\rho_j \cap N$.  

We construct the toric variety $V_\Sigma$ as follows.  To each edge $\rho_j \in \Sigma(1)$, we associate a coordinate $z_j$.  Let $\mathcal{S}$ denote any subset of $\Sigma(1)$ that does \emph{not} span a cone of $\Sigma$.  Let $\mathcal{V}(\mathcal{S})\subseteq \C^q$ be the linear subspace defined by setting $z_j = 0$ for each $\rho_j\in \mathcal{S}$.  Let $Z(\Sigma)$ be the union of the spaces $\mathcal{V}(\mathcal{S})$.  Note that $(\C^*)^q$ acts on $\C^q - Z(\Sigma)$ by coordinate-wise multiplication.  Fix a basis for $N$, and suppose that $v_j$ has coordinates $(v_{j 1}, \dots, v_{j n})$ with respect to this basis.  Consider the map $\phi : (\C^*)^q \to (\C^*)^n$ given by
\[ \phi(t_1, \dots, t_q) \mapsto \left( \prod_{j=1}^n t_j ^{v_{j1}} , \dots, \prod_{j=1}^n t_j^{v_{j n}} \right) \]
Then the toric variety $V_\Sigma$ associated with the fan $\Sigma$ is given by
\[ V_\Sigma  = (\C^q - Z(\Sigma)) / \text{Ker}(\phi).\]

Given a lattice polytope $\diamond$ in $N$, we define its \textit{polar polytope} $\diamond^\circ$ to be $\diamond^\circ = \{w \in $M$ \, | \, \langle v , w \rangle \geq -1 \, \forall \, v \in K\}$.  If $\diamond^\circ$ is also a lattice polytope, we say that $\diamond$ is a reflexive polytope and that $\diamond$ and $\diamond^\circ$ are a mirror pair.  

\example{The generalized octahedron in $N$ with vertices at $(\pm1,0,\dots,0)$, $(0,\pm1,\dots,0),$ $\dots,$ $(0,0,\dots,\pm1)$ is a reflexive polytope.  Its polar polytope is the hypercube with vertices at $(\pm1,\pm1,\dots,\pm1)$.}

A reflexive polytope must contain $\vec{0}$; furthermore, $\vec{0}$ is the only interior lattice point of the polytope.  We may obtain a fan $R$ by taking cones over the faces of $\diamond$.  Let $\Sigma$ be a simplicial refinement of $R$ such that the one-dimensional cones of $\Sigma$ are generated by the nonzero lattice points $v_k$, $k = 1 \dots q$, of $\diamond$; we call such a refinement a \emph{maximal projective subdivision}.  Then the variety $V_\Sigma$ is an orbifold; if $n=3$, $V_\Sigma$ is smooth (see \cite{CoxKatz}).  

\begin{example}\label{E:octahedron}
Let $N \cong \Z^3$, and let $\diamond$ be the octahedron with vertices $v_1=(1,0,0)$, $v_2=(0,1,0)$, $v_3=(0,0,1)$, $v_4=(-1,0,0)$, $v_5=(0,-1,0)$, and $v_6=(0,0,-1)$.  Then the only lattice points of $\diamond$ are the vertices and the origin.  Let $R$ be the fan obtained by taking cones over the faces of $\diamond$.  Then $R$ defines a toric variety $V_R$ which is isomorphic to $\P^1 \times \P^1 \times \P^1$.
\end{example}

\begin{figure}[h!]
\begin{center}
\scalebox{.5}{\includegraphics{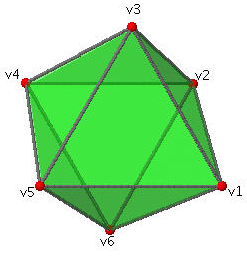}}
\end{center}
\caption{The octahedron of Example~\ref{E:octahedron}}
\end{figure}

\begin{proof}
The vertices of the octahedron $v_1, \dots, v_6$ generate the one-dimensional cones $\rho_1, \dots, \rho_6$ of $R$.  The two-element subsets of $\Sigma(1)$ that do not span cones are $\{\rho_1,\rho_4\}$, $\{\rho_2, \rho_5\}$, and $\{\rho_3, \rho_6\}$; larger subsets of $\Sigma(1)$ that do not span cones contain one of the two-element subsets that do not span cones.  Thus, $Z(\Sigma)$ consists of points of the form $(0,z_2,z_3,0,z_5,z_6)$, $(z_1,0,z_3,z_4,0,z_6)$, or $(z_1,z_2,0,z_4,z_5,0)$.

The map $\phi$ is given by:
$$
\phi(t_1,t_2,t_3,t_4,t_5,t_6) = (t_1t_4^{-1},t_2t_5^{-1},t_3t_6^{-1})
$$
Then $V_R$ is given by the quotient $\C^6\backslash Z(\Sigma)/\ker(\phi)$, where $\ker(\phi)$ contains points satisfying $t_1=t_4,t_2=t_5$, and $t_3=t_6$.  This corresponds to the equivalence relations
\begin{align*}
(z_1,z_2,z_3,z_4,z_5,z_6) &\sim (\lambda_1 z_1,z_2,z_3,\lambda_1 z_4,z_5,z_6)\\
(z_1,z_2,z_3,z_4,z_5,z_6) &\sim (z_1,\lambda_2 z_2,z_3,z_4,\lambda_2 z_5,z_6)\\
(z_1,z_2,z_3,z_4,z_5,z_6)&\sim (z_1,z_2,\lambda_3 z_3,z_4,z_5,\lambda_3 z_6)
\end{align*}
where $\lambda_1,\lambda_2,\lambda_3\in \C^*$. Thus, $V_R$ is isomorphic to the toric variety $\P\times \P\times \P$.
\end{proof}

\begin{example}\label{E:skewOctahedron}
Let $N \cong \Z^3$, and let $\diamond$ be the octahedron with vertices $v_1 = (1,0,0)$, $v_2 = (1,2,0)$, $v_3 = (1,0,2)$, $v_4 =(-1,0,0)$, $v_5 = (-1,-2,0)$, and $v_6 = (-1,0,-2)$.  Let $R$ be the fan obtained by taking cones over the faces of $\diamond$.  Then $R$ defines a toric variety $V_R$ which is isomorphic to $(\P\times \P\times \P)/(\Z_2\times \Z_2\times \Z_2)$.   If $\Sigma$ is a simplicial refinement of $R$ such that the one-dimensional cones of $\Sigma$ are generated by the nonzero lattice points of $\diamond$, then $V_\Sigma$ is a smooth variety and the map $V_\Sigma \to V_R$ is a resolution of singularities.
\end{example}

\begin{figure}[h!]
\begin{center}
\scalebox{.5}{\includegraphics{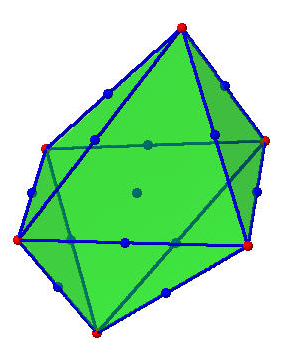}}
\end{center}
\caption{The octahedron of Example~\ref{E:skewOctahedron}}
\end{figure}

\begin{proof}
As in Example~\ref{E:octahedron}, $Z(\Sigma)$ consists of points of the form $(0,z_2,z_3,0,z_5,z_6)$, $(z_1,0,z_3,z_4,0,z_6)$, or $(z_1,z_2,0,z_4,z_5,0)$.  The map $\phi$ is defined as
$$
\phi(t_1,t_2,t_3,t_4,t_5,t_6) = (t_1t_2t_3t_4^{-1}t_5^{-1}t_6^{-1}, t_2^2 t_5^{-2}, t_3^2 t_6^{-2}).
$$
Thus, elements of $\ker(\phi)$ must satisfy $t_1t_2t_3 = t_4t_5t_6$, $t_2^2 = t_5^{2}$, and $t_3^2 = t_6^{2}$. These equations simplify to $t_1^2 = t_4^2, t_2^2 = t_5^2$ and $t_3^2 = t_6^2$.  We obtain the equivalence relations
\begin{align*}
(z_1,z_2,z_3,z_4,z_5,z_6) &\sim (\lambda_1 z_1,z_2,z_3,\pm\lambda_1 z_4,z_5,z_6)\\
(z_1,z_2,z_3,z_4,z_5,z_6) &\sim (z_1,\lambda_2 z_2,z_3,z_4,\pm\lambda_2 z_5,z_6)\\
(z_1,z_2,z_3,z_4,z_5,z_6)&\sim (z_1,z_2,\lambda_3 z_3,z_4,z_5,\pm\lambda_3 z_6)
\end{align*}
where $\lambda_1,\lambda_2,\lambda_3\in \C^*$.  We conclude that $V_R$ is isomorphic to $(\P^1\times \P^1\times \P^1)/(\Z_2\times \Z_2\times \Z_2)$.  Since $n=3$, the simplicial refinement $\Sigma$ yields a smooth variety.
\end{proof}

\subsection{Semiample hypersurfaces and the residue map}\label{SS:residue}

In this section, we review properties of hypersurfaces in toric varieties, and give a brief outline of the results of \cite{Mavlyutov} on the residue map in this setting.  Let $\Sigma$ be a complete, simplicial $n$-dimensional fan, and let $S = \mathbb{C}[z_1, \dots, z_q]$ be the homogeneous coordinate ring of the corresponding toric variety $V_\Sigma$.  Each variable $z_i$ defines an irreducible torus-invariant divisor $D_i$, given by the points where $z_i = 0$.  The homogeneous coordinate ring is graded by the Chow group of $V_\Sigma$, according to the rule 
\[\mathrm{deg}(\prod_{i=1}^n x_i^{a_i}) = \sum_{i=1}^n a_i D_i.\]
A homogeneous polynomial $p$ in $S_\beta$ defines a hypersurface $X$ in $V_\Sigma$.  

\begin{definition}\cite{BC}
If the products $\partial p/\partial z_i$, $i = 1 \dots q$ do not vanish simultaneously on $X$, we say $X$ is \emph{quasismooth}.
\end{definition}

\begin{definition}\cite{Mavlyutov}If the products $z_i \, \partial p/\partial z_i$, $i = 1 \dots q$ do not vanish simultaneously on $X$, we say $X$ is \emph{regular} and $p$ is \emph{nondegenerate}. 
\end{definition}

Let $R$ be a fan over the faces of a reflexive polytope, and assume $\Sigma$ is a refinement of $R$.  We have a proper birational morphism $\pi: V_\Sigma \to V_R$.  Let $Y$ be an ample divisor in $V_R$, and suppose $X = \pi^*(Y)$.  Then $X$ is semiample:

\begin{definition}\cite[Lemma 4.1.2]{CoxKatz}
We say that a Cartier divisor $D$ is \emph{semiample} if $D$ is generated by global sections and the intersection number $D^n > 0$.
\end{definition}

Note that if $\Sigma$ is not identical to $R$, then $X$ is not ample.  If $\Sigma$ is a maximal projective subdivision of $R$, then general representatives $X$ of the anticanonical class of $V_\Sigma$ are Calabi-Yau varieties; if $n=3$, then the representatives are K3 surfaces.  (See \cite{CoxKatz} and \cite[\S 1]{Mavlyutov} for a more detailed exposition.) 

Now, let us assume that $X$ is a semiample, quasismooth hypersurface defined by a polynomial $p \in S_\beta$.  The \emph{residue map} relates the cohomology of $V_\Sigma - X$ to the cohomology of $X$:
\[\mathrm{Res}: H^n(V_\Sigma - X) \to H^{n-1}(X).\]

In order to give a precise definition of the residue map, let us represent elements of $H^n(V_\Sigma - X)$ using rational forms.  Choose an integer basis $m_1,\dots,m_n$ for the dual lattice $M$.  For any $n$-element subset $I = \{i_1, \dots, i_n\}$ of $\{1,\dots,q\}$, let $\mathrm{det}\,v_I = \mathrm{det}\,(\langle m_j, v_{i_k} \rangle_{1 \leq j, i_k \leq n})$, $dz_I = dz_{i_1} \wedge \dots \wedge dz_{i_n}$, and $\hat{z}_I = \prod_{i \notin I}z_i$.  Let $\Omega$ be the $n$-form on $V_\Sigma$ given in global homogeneous coordinates by $\sum_{|I|=n}\mathrm{det}\,v_I \hat{z}_I dz_I$.  (Note that if $V_\Sigma = \mathbb{P}^n$, then $\Omega$ is the usual holomorphic form on $\mathbb{P}^n$.)  Let $\beta_0 = \sum_{i=1}^n \mathrm{deg}(x_i)$, and let $A \in S_{(a+1) \beta - \beta_0}$.  Then the rational form $\omega_A := \frac{A \Omega}{p^{a+1}}$ is a class in $H^n(V_\Sigma-X)$.  Let $\gamma$ be any $n-1$-cycle in $X$, and let $T(\gamma)$ be the tube over $\gamma$ in $V_\Sigma - X$.  Then the residue of $\omega_A$ is the class in $H^{n-1}(X)$ satisfying
\begin{equation}
\int_{\gamma} \mathrm{Res}\left(\frac{A \Omega}{p^{a+1}}\right) = \int_{T(\gamma)} \frac{A \Omega}{p^{a+1}}.
\end{equation}
The residue class $\mathrm{Res}(\omega_A)$ lies in $H^{n-1-a,a}(X)$.  (See \cite[\S 3]{Mavlyutov}.)  We shall have occasion to use the following special case of this construction:

\begin{lemma}\label{P:Mav}\cite[\S 3]{Mavlyutov}
Let $X$ be a quasismooth K3 hypersurface in $V_\Sigma$ described in global homogeneous coordinates by a polynomial $p$.  Then $\omega := \mathrm{Res}(\Omega/p)$ generates $H^{2,0}(X)$.
\end{lemma}

We may use rational forms and the residue map to relate $H^{n-1}(X)$ to certain quotient rings.

\begin{definition}\cite{BC}
Let $p \in S_\beta$.  Then the \emph{Jacobian ideal} $J(p)$ is the ideal of $S$ generated by the partial derivatives $\partial p/\partial z_i$, $i = 1 \dots q$, and the \emph{Jacobian ring} $R(p)$ is the quotient ring $S/J(p)$.  The Jacobian ring inherits a grading from $S$.
\end{definition}

\begin{proposition}\cite[\S 3]{Mavlyutov}
If $X$ is a quasismooth, semiample hypersurface defined by a polynomial $p$, then the residue map induces a well-defined map of rings 
\[\mathrm{Res}_J: R(p) \to H^{n-1}(X)\]
satisfying $\mathrm{Res}_J([A]_{R(p)}) = \mathrm{Res}(\omega_A)$.
\end{proposition}

If $V_\Sigma$ is isomorphic to $\mathbb{P}^n$ or a weighted projective space, then the map $\mathrm{Res}_J$ is injective.  (See \cite[\S 11]{BC}.)  One may obtain injective maps for more general ambient spaces by working with a different quotient ring; however, these results only apply when the hypersurface $X$ is regular.

\begin{definition}\cite{BC}
Let $p \in S_\beta$.  Then the ideal $J_1(p)$ is the ideal quotient
\[\langle z_1 \partial p/\partial z_1, \dots, z_q \partial p/\partial z_q \rangle : z_1 \cdots z_q. \]
The ring $R_1(p)$ is the quotient ring $S/J_1(p)$; this ring inherits a grading from $S$.
\end{definition}

\begin{theorem}\cite[Theorem 4.4]{Mavlyutov}
If $X$ is a regular, semiample hypersurface defined by a polynomial $p$, then the residue map induces a well-defined, injective map of rings 
\[\mathrm{Res}_{J_1}: R_1(p) \to H^{n-1}(X)\]
satisfying $\mathrm{Res}_{J_1}([A]_{R_1(p)}) = \mathrm{Res}(\omega_A)$.
\end{theorem}

\section{Three symmetric families of K3 surfaces}

\subsection{Symplectic group actions on K3 surfaces}

Let $X$ be a K3 surface and let $g$ be an automorphism of $X$.  We say that $g$ \emph{acts symplectically} if $g^*(\omega) = \omega$, where $\omega$ is the unique holomorphic two-form on $X$.  If $G$ is a finite group of automorphisms of $X$, we say $G$ \emph{acts symplectically} on $X$ if every element of $G$ acts symplectically.

The cup product induces a bilinear form $\langle \, , \rangle$ on $H^2(X,\mathbb{Z}) \cong H \oplus H \oplus H \oplus E_8 \oplus E_8$.  (We take $E_8$ to be negative definite.)  Using this form, we define $S_G = (H^2(X,\mathbb{Z})^G)^\perp$.  The Picard group of $X$, $\mathrm{Pic}(X)$,  consists of $H^{1,1}(X) \cap H^2(X,\mathbb{Z})$; the group $\mathcal{T}(X) \subseteq H^2(X,\mathbb{Z})$ of transcendental cycles is defined as $(\mathrm{Pic}(X))^\perp$.  Nikulin showed that the groups $\mathrm{Pic}(X)$ and $S_G$ are related:

\begin{proposition}\label{P:Nik} \cite[Lemma 4.2]{Nikulin}
$S_G \subseteq \mathrm{Pic}(X)$ and $\mathcal{T}(X) \subseteq H^2(X,\mathbb{Z})^G$.  The lattice $S_G$ is nondegenerate and negative definite.
\end{proposition}

The rank of the lattice $S_G$ depends only on the group $G$.  \cite[Table 2]{Xiao} lists the rank of $S_G$ for each group $G$ which admits a symplectic action on a K3 surface; a discussion of methods for computing lattice invariants of $S_G$ may be found in \cite{CommAlg}.

\begin{lemma}\cite[Example 2.1]{CommAlg}\label{L:S4}
Let $X$ be a K3 surface which admits a symplectic action by the permutation group $G = \mathcal{S}_4$.  Then $\mathrm{Pic}(X)$ admits a primitive sublattice $S_G$ which has rank $17$ and discriminant $d(S_G) = -2^6 \cdot 3^2$.
\end{lemma}

\subsection{An $\mathcal{S}_4$ symmetry of polytopes and hypersurfaces}\label{S:toricexamples}

Let $\diamond$ be a reflexive polytope in a lattice $N \cong \mathbb{Z}^3$, and let $\Sigma$ be a simplicial refinement of the fan over the faces of $\diamond$.  Demazure and Cox showed that the automorphism group $A$ of the toric variety $V_\Sigma$ is generated by the big torus $T \cong (\mathbb{C}^*)^3$, symmetries of the fan $\Sigma$ induced by lattice automorphisms, and one-parameter families derived from the ``roots'' of $V_\Sigma$ (see \cite{CoxKatz}).  We are interested in finite subgroups of $A$ which act symplectically on K3 hypersurfaces $X$ in $V_\Sigma$. 

Let us consider the automorphisms of $V_\Sigma$ induced by symmetries of the fan $\Sigma$.  Since $\Sigma$ is a refinement of the fan $R$ consisting of cones over the faces of $\diamond$, the group of symmetries of $\Sigma$ must be a subgroup $H'$ of the group $H$ of symmetries of $\diamond$ (viewed as a lattice polytope).  We will identify a family $\mathcal{F}_\diamond$ of K3 surfaces in $V_\Sigma$ on which $H'$ acts by automorphisms, and then compute the induced action of $G$ on the $(2,0)$ form of each member of the family.  

Let $h \in H'$, and let $X$ be a K3 surface in $V_\Sigma$ defined by a polynomial $p$ in global homogeneous coordinates.  Then $h$ maps lattice points of $\diamond$ to lattice points of $\diamond$, so we may view $h$ as a permutation of the global homogeneous coordinates $z_i$: $h$ is an automorphism of $X$ if $p \circ h = p$.  Alternatively, since $H$ is the automorphism group of both $\diamond$ and its polar dual polytope $\diamond^\circ$, we may view $h$ as an automorphism of $\diamond^\circ$: from this vantage point, we see that $h$ acts by a permutation of the coefficients $c_x$ of $p$, where each coefficient $c_x$ corresponds to a point $x \in \diamond^\circ$.  Thus, if $h$ is to preserve $X$, we must have $c_x = c_y$ whenever $h(x) = y$.  We may define a family of K3 surfaces fixed by $H'$ by requiring that $c_x = c_y$ for any two lattice points $x,y \in \diamond^\circ$ which lie in the same orbit of $H'$:

\begin{proposition}\label{P:FanSymmetryFamily}
Let $\mathcal{F}_\diamond$ be the family of K3 surfaces in $V_\Sigma$ defined by the following family of polynomials in global homogeneous coordinates:
\[ p = (\sum_{Q \in \mathscr{O}} c_{Q} \sum_{x \in Q} \prod_{k=1}^q z_k^{\langle v_k, x \rangle + 1}) + \prod_{k=1}^q z_k,\]
where $\mathscr{O}$ is the set of orbits of nonzero lattice points in $\diamond^\circ$ under the action of $H'$.  Then $H'$ acts by automorphisms on each K3 surface $X$ in $\mathcal{F}_\diamond$.
\end{proposition}

\begin{proposition}
Let $X$ be a quasismooth K3 surface in the family $\mathcal{F}_\diamond$, and let $h \in H'  \subset \mathbf{GL}(3,\mathbb{Z})$.  Then $h^*(\omega) = (\mathrm{det}~h)\omega$.
\end{proposition}

\begin{proof}
Once again, we use the fact that we may view $h$ as either an automorphism of the lattice $N$ which maps $\diamond$ to itself, or as an automorphism of the dual lattice $M$ which restricts to an automorphism of $\diamond^\circ$.  (If we fix a basis $\{n_1,n_2,n_3\}$ of $N$, take the dual basis $\{m_1,m_2,m_3\} = \{n_1^*,n_2^*,n_3^*\}$ on $M$, and treat $h$ as a matrix, then $h$ acts on $M$ by the inverse matrix.)  By Proposition~\ref{P:Mav}, each choice of basis for $M$ yields a generator of $H^{3,0}(V)$.  Thus, if $\Omega$ is the generator of $H^{3,0}(V)$ corresponding to a fixed choice of integer basis $m_1,m_2,m_3$, we see that we may obtain a new generator $\Omega'$ of $H^{3,0}(V)$ by applying the change of basis $h^{-1}$ to $M$.  Recall that $\Omega = \sum_{|I|=3}\mathrm{det}\,v_I \hat{z}_I dz_I$, where $\mathrm{det}\,v_I = \mathrm{det}\,(\langle m_j, v_{i_k} \rangle_{1 \leq j, i_k \leq 3})$.

We compute: 

\begin{align}
\Omega' &= \sum_{|I|=3}\mathrm{det}\,(h^{-1}(v_I)) \hat{z}_I dz_I \\
&= \sum_{|I|=3}\mathrm{det}\,(h^{-1})\mathrm{det}\,v_I \hat{z}_{I} dz_{I}\\
&= \mathrm{det}\,h \sum_{|I|=3}\mathrm{det}\,v_I \hat{z}_{I} dz_{I}
\end{align}  
since $\mathrm{det}\,h = \pm 1$.

By Proposition~\ref{P:FanSymmetryFamily}, $h^*(p) = p$, so $h^*(\omega) = \mathrm{Res}(\Omega'/p) = (\mathrm{det}\,h) \omega$.
\end{proof}

Thus the group $G$ of orientation-preserving automorphisms of $\diamond$ which preserve $\Sigma$ acts symplectically on quasismooth members of $\mathcal{F}_\diamond$.

The largest group which occurs as the orientation-preserving automorphism group of a three-dimensional lattice polytope is $S_4$.  There are three distinct pairs of isomorphism classes of reflexive polytopes which have this symmetry group.  In the following examples, we analyze families derived from these pairs of polytopes.

\begin{example}\label{E:CubeAndOctahedron}
Let $\diamond$ be the cube with vertices of the form $(\pm 1, \pm 1, \pm 1)$.  The dual polytope $\diamond^\circ$ is an octahedron, with vertices $\{(\pm 1,0,0), (0,\pm 1, 0), (0,0,\pm 1)\}$.  We may choose our fan $\Sigma$ such that the group of lattice automorphisms of $\diamond$ preserves $\Sigma$.  The group $G$ of orientation-preserving automorphisms of $\diamond$ is isomorphic to $S_4$.  $\mathcal{F}_\diamond$ is a one-parameter family, and if $X$ is a quasismooth member of $\mathcal{F}_\diamond$, $\mathrm{rank} \; \mathrm{Pic}(X) \geq  19$.
\end{example}

\begin{proof}
The action of $G$ on $\diamond^\circ$ has two orbits: the origin, and the vertices of the octahedron.  Thus, $\mathcal{F}_\diamond$ is a one-parameter family.  Using Lemma~\ref{L:S4}, we conclude that for any quasismooth member of $\mathcal{F}_\diamond$, $\mathrm{rank}\,S_G = 17$. 

Let $X$ be a quasismooth member of $\mathcal{F}_\diamond$.  We wish to determine which of the divisors of $X$ inherited from the ambient toric variety $V_\Sigma$ are in $H^2(X,\mathbb{Z})^G$.  The action of $G$ on the lattice points of $\diamond$ has four orbits: the origin, the vertices of the cube, the interior points of edges, and interior points of faces.  Let $v_1, \dots, v_8$ be the vertices of the cube and $v_9, \dots, v_{20}$ be the interior points of edges; let $W_1, \dots, W_{20}$ be the corresponding torus-invariant divisors of the toric variety $V_\Sigma$.  Since $v_1, \dots, v_8$ and $v_9, \dots, v_{20}$ are orbits of the action of $G$, $W_1 + \dots + W_8$ and $W_9 + \dots + W_{20}$ are elements of $\mathrm{Pic}(V)$ which are fixed by $G$.  These two divisors span a rank-two lattice in $\mathrm{Pic}(V)$.  Since there are no lattice points strictly in the interior of the edges of $\diamond^\circ$ and none of the points $v_1, \dots, v_{20}$ lies in the relative interior of a facet of $\diamond$, $W_k \cap X$ is connected and nonempty for $1 \leq k \leq 20$ and the divisors $W_1 \cap X + \dots + W_8 \cap X$ and $W_9 \cap X + \dots + W_{20} \cap X$ span a rank-two lattice in $\mathrm{Pic}(X)$.  This rank-two lattice is contained in $H^2(X,\mathbb{Z})^G$.

Since $S_G$ is the perpendicular complement of $H^2(X,\mathbb{Z})^G$, $\mathrm{rank} \; \mathrm{Pic}(X) \geq  17 + 2 = 19$.
\end{proof}

\begin{remark}
This family is analyzed in \cite{PS} and \cite{HLOY}.
\end{remark}

\begin{example}
Let $\diamond$ be a three-dimensional reflexive polytope with fourteen vertices and twelve faces.  Up to lattice isomorphism, $\diamond$ is unique; moreover, $\diamond$ has the most vertices of any three-dimensional reflexive polytope.  We may choose our fan $\Sigma$ such that the group of lattice automorphisms of $\diamond$ preserves $\Sigma$.  The group $G$ of orientation-preserving automorphisms of $\diamond$ is isomorphic to $S_4$, and $\mathcal{F}_\diamond$ is a one-parameter family.  If $X$ is a quasismooth member of $\mathcal{F}_\diamond$, $\mathrm{rank} \; \mathrm{Pic}(X) \geq  19$.
\end{example}

\begin{proof}
The lattice points of $\diamond^\circ$ consist of vertices and the origin, and $G$ acts transitively on the vertices of $\diamond^\circ$, so $\mathcal{F}_\diamond$ is a one-parameter family.  As above, Lemma~\ref{L:S4} shows that for any quasismooth member of $\mathcal{F}_\diamond$, $\mathrm{rank}\,S_G = 17$. 

Let $X$ be a quasismooth member of $\mathcal{F}_\diamond$.  Once again, we determine which of the divisors of $X$ inherited from the ambient toric variety $V_\Sigma$ are in $H^2(X,\mathbb{Z})^G$.  The action of $G$ on the lattice points of $\diamond$ has three orbits; one orbit contains the origin, another contains eight vertices, and the last contains the remaining six vertices.  Let $\{v_1, \dots, v_8\}$ and $\{v_9, \dots, v_{14}\}$ be the vertex orbits; let $W_1, \dots, W_{14}$ be the corresponding torus-invariant divisors of $V_\Sigma$.  Then $W_1 + \dots + W_8$ and $W_9 + \dots + W_{14}$ are elements of $\mathrm{Pic}(V)$ fixed by the action of $G$; these two divisors span a rank-two lattice in $\mathrm{Pic}(V)$.  Since there are no lattice points strictly in the interior of the edges of $\diamond^\circ$ and the facets of $\diamond$ have no points in their relative interiors, $W_k \cap X$ is connected and nonempty for $1 \leq k \leq 14$ and the divisors $W_1 \cap X + \dots + W_8 \cap X$ and $W_9 \cap X + \dots + W_{14} \cap X$ span a rank-two lattice in $\mathrm{Pic}(X)$.  This rank-two lattice is contained in $H^2(X,\mathbb{Z})^G$, so $\mathrm{rank} \; \mathrm{Pic}(X) \geq  17 + 2 = 19$. 

\end{proof}

\begin{remark}
An explicit analysis of the same family appears in \cite{Verrill}.
\end{remark}

\begin{example}\label{E:OctahedronAndSkewCube}
Let $\diamond$ be the octahedron with vertices $(1, 1, 1)$, $(-1, -1, 1)$, $(-1, 1, -1)$, $(1, -1, 1)$, $(1, 1, -1)$, and $(-1, -1, -1)$.  The polar dual $\diamond^\circ$ has vertices $(1, 0, 0)$, $(0, 1, 0)$, $(0, 0, 1)$, $(-1, 1, 1), (1, -1, -1)$, $(0, 0, -1)$, $(0, -1, 0)$, and $(-1, 0, 0)$.  We may choose our fan $\Sigma$ such that the group of lattice automorphisms of $\diamond$ preserves $\Sigma$.  The group $G$ of orientation-preserving automorphisms of $\diamond$ is isomorphic to $S_4$.  $\mathcal{F}_\diamond$ is a one-parameter family.  If $X$ is a quasismooth member of $\mathcal{F}_\diamond$, $\mathrm{rank} \; \mathrm{Pic}(X) \geq  19$.
\end{example}

\begin{proof}
The action of $G$ on $\diamond^\circ$ has two orbits, the origin and the polytope's vertices, so $\mathcal{F}_\diamond$ is a one-parameter family.  As in the previous example, Lemma~\ref{L:S4} shows that for any quasismooth member of $\mathcal{F}_\diamond$, $\mathrm{rank}\,S_G = 17$. 

Let $X$ be a quasismooth member of $\mathcal{F}_\diamond$.  As before, we determine which of the divisors of $X$ inherited from the ambient toric variety $V_\Sigma$ are in $H^2(X,\mathbb{Z})^G$.  The action of $G$ on the lattice points of $\diamond$ has three orbits: the origin, the octahedron's vertices, and the interior points of edges.  Let $v_1, \dots, v_6$ be the vertices and $v_7, \dots, v_{18}$ be the interior points of edges; let $W_1, \dots, W_{18}$ be the corresponding torus-invariant divisors of $V_\Sigma$.  Then $W_1 + \dots + W_6$ and $W_7 + \dots + W_{18}$ are elements of $\mathrm{Pic}(V)$ fixed by the action of $G$.  These two divisors span a rank-two lattice in $\mathrm{Pic}(V)$.  Since there are no lattice points strictly in the interior of the edges of $\diamond^\circ$ and the facets of $\diamond$ have no points in their relative interiors, $W_k \cap X$ is connected and nonempty for $1 \leq k \leq 18$ and the divisors $W_1 \cap X + \dots + W_6 \cap X$ and $W_7 \cap X + \dots + W_{18} \cap X$ span a rank-two lattice in $\mathrm{Pic}(X)$.  This rank-two lattice is contained in $H^2(X,\mathbb{Z})^G$.  Thus, $\mathrm{rank} \; \mathrm{Pic}(X) \geq  17 + 2 = 19$. 
\end{proof}


\section{Picard-Fuchs Equations}

\subsection{The Griffiths-Dwork technique}

A \emph{period} is the integral of a differential form with respect to a specified homology class.  The \emph{Picard-Fuchs differential equation} of a family of varieties is a differential equation which describes the way the value of a period changes as we move through the family.  We may use Picard-Fuchs differential equations for periods of holomorphic forms to understand the way the complex structure of a family of varieties varies within the family.  The \emph{Griffiths-Dwork technique} provides an algorithm for computing Picard-Fuchs equations for families of hypersurfaces in projective space.  This technique has been generalized to hypersurfaces in weighted projective space and in some toric varieties.  Unlike other methods for computing Picard-Fuchs equations, the Griffiths-Dwork technique allows the study of arbitrary rational parametrizations.  

Let us begin by reviewing the Griffiths-Dwork technique for one-parameter families of hypersurfaces $X_t$ in $\mathbb{P}^n$ described by homogeneous polynomials $p_t$ of degree $\ell$.  We may define a flat family
of cycles $\gamma_t$. 
We then differentiate as follows:

\begin{align}
\frac{d}{dt} \int_{\gamma_t} \mathrm{Res}\left(\frac{A \Omega}{p_t^k}\right) &= \int_{\gamma_t} \mathrm{Res}\left(\frac{d}{dt}\left(\frac{A \Omega}{p_t^k}\right)\right) \\
& = -k \int_{\gamma_t} \mathrm{Res}\left(\frac{(\frac{dp_t}{dt}) A \Omega}{p_t^{k+1}}\right). \notag
\end{align}

Thus, we may express successive derivatives of the period $\int_{\gamma_t} \frac{\Omega}{p}$ as periods of the residues of rational forms. If $H^{n-1}(X,\mathbb{C})$ is $r$-dimensional as a vector space over $\mathbb{C}$, then at most $r$ residues of rational forms can be
linearly independent. Therefore, the period must satisfy a linear differential equation with coefficients in
$\mathbb{C}(t)$ of order at most $r$; this linear differential equation is the Picard--Fuchs differential equation which we seek.  

In order to compute the Picard-Fuchs differential equation in practice, we need a way to compare expressions of the form $\mathrm{Res}\left(\frac{A \Omega}{p_t^{k}}\right)$ to expressions of the form $\mathrm{Res}\left(\frac{B \Omega}{p_t^{k+1}}\right)$.
Suppose we have an element of $H^{n-1}(X,\mathbb{C})$ of the form $\mathrm{Res}\left(\frac{K \Omega_0}{p^{k+1}}\right)$, where $K = \sum_i A_i \frac{\partial p}{\partial x_i}$ is a member of the Jacobian ideal, and each $A_i$ is a homogeneous polynomial of degree $k\cdot \ell -n$.
Then the following equation allows us to reduce the order of the pole:

\begin{equation}\label{E:reducePoleOrder}
\frac{\Omega_0}{p^{k+1}} \sum_i A_i \frac{\partial p}{\partial x_i} = \frac{1}{k} \frac{\Omega_0}{p^{k}} \sum_i \frac{\partial A_i}{\partial x_i} + \mathrm{exact}\;\mathrm{terms}
\end{equation}

We may find the Picard-Fuchs equation by systematically taking derivatives of $\int_{\gamma_t}\mathrm{Res}\left(\frac{\Omega_0}{p}\right)$ and using \ref{E:reducePoleOrder} to
rewrite the results in terms of a standard basis for $H^{n-1}(X,\mathbb{C})$.  This method is known as the \emph{Griffiths-Dwork technique}.  Practical implementations of the Griffiths-Dwork technique use the Jacobian ring $J(p)$ and the induced residue map $\mathrm{Res}_J$ to transform the problem into a computation suitable for a computer algebra system.  (See \cite{CoxKatz} or \cite{DGJ} for a more detailed discussion of the technique.)

In order to extend the Griffiths-Dwork technique to hypersurfaces in toric varieties, we need two tools: an appropriate version of the residue map, and an analogue of Equation~\ref{E:reducePoleOrder} to reduce the order of the poles.  In the case of semiample hypersurfaces in toric varieties, we may use the results of \cite{BC} and \cite{Mavlyutov} described in \S~\ref{SS:residue} to define $\mathrm{Res}$.  We must be aware, however, that the induced residue map $\mathrm{Res}_J$ need not be injective for an arbitrary family of semiample hypersurfaces.  

To construct an analogue of Equation~\ref{E:reducePoleOrder}, we note that the results of \cite{BC} apply in the semiample case:
\begin{definition}\cite[Definition 9.8]{BC}
Let $i \in \{1, \dots, q\}$.  We define the $n-1$-form $\Omega_i$ on $V_\Sigma$ as follows:
\[\Omega_i = \sum_{|J| = n-1, i \notin J} \mathrm{det}(v_{\{i\} \cup J }) \hat{z}_{\{i\} \cup J} dz_J.\]
Here we use the convention that $i$ is the first element of $\{i\} \cup J$.
\end{definition}

\begin{lemma}\cite[Lemma 10.7]{BC}
If $A \in S_{k\beta -\beta_0 + \beta_i}$, then:
\[d\left( \frac{A \Omega_i}{p^{k}}\right) = \frac{\left( p \frac{\partial A}{\partial z_i} - k A \frac{\partial p}{\partial z_i}\right) \Omega_0}{p^{k+1}}.\]
\end{lemma}

Now, let $X$ be a hypersurface in a toric variety $V_\Sigma$ described by a homogeneous polynomial $p \in S_\beta$.  Suppose we have an element of $H^{n-1}(X,\mathbb{C})$ of the form $\mathrm{Res}\left(\frac{K \Omega_0}{p^{k+1}}\right)$, where $K = \sum_i A_i \frac{\partial p}{\partial x_i}$ is a member of the Jacobian ideal, and $A_i \in S_{k\beta -\beta_0 + \beta_i}$.  The following reduction of pole order equation follows immediately:

\begin{equation}\label{E:reducePoleOrderToric}
\frac{\Omega_0}{p^{k+1}} \sum_i A_i \frac{\partial p}{\partial x_i} = \frac{1}{k} \frac{\Omega_0}{p^{k}} \sum_i \frac{\partial A_i}{\partial x_i} + \mathrm{exact}\;\mathrm{terms}
\end{equation}

\subsection{A Picard-Fuchs equation}


Let $\diamond$ be the octahedron with vertices $(1, 1, 1)$, $(-1, -1, 1)$, $(-1, 1, -1)$, $(1, -1, 1)$, $(1, 1, -1)$, and $(-1, -1, -1)$, as in Example~\ref{E:OctahedronAndSkewCube}, and let $\mathcal{F}_\diamond$ be the associated one-parameter family.  In this section, we describe the Picard-Fuchs equation for $\mathcal{F}_\diamond$.  We use our result to show that the Picard rank of a general member of $\mathcal{F}_\diamond$ is exactly $19$. 

Doran analyzed the properties of Picard-Fuchs equations for lattice-polarized families of K3 surfaces with Picard rank $19$ in \cite{Doran}, and showed that the Picard-Fuchs equations for the K3 surfaces are related to Picard-Fuchs equations for families of elliptic curves.

\begin{proposition}\cite[Lemma 3.1.(b)]{Singer}
Let $L(y)$ be a homogeneous linear differential polynomial with coefficients in ${\mathbb{C}}(t)$.  Then there exists a homogeneous linear differential equation $M(y) = 0$ with coefficients in $\mathbb{C}(t)$ and solution space the $\mathbb{C}$-span of 
$$\{\nu_1 \nu_2 \ | \ L(\nu_1) = 0 \ \mbox{and} \ L(\nu_2) = 0 \} \ .$$
\end{proposition}

\begin{definition}
We call the operator $M(y)$ constructed above the {\em symmetric square} of $L$.
\end{definition}

The symmetric square of the second-order linear, homogeneous differential equation
$$
a_2\derive{\omega}{2} + a_1\frac{\partial \omega}{\partial t} + a_0\omega = 0
$$
is
\begin{equation}\label{E:symmetricSquare}
a_2^2\derive{\omega}{3} + 3a_1a_2\derive{\omega}{2} + (4a_0a_2+2a_1^2+a_2a_1'-a_1a_2')\frac{\partial \omega}{\partial t} + (4a_0a_1 + 2a_0'a_2-2a_0a_2')\omega = 0
\end{equation}
where primes denote derivatives with respect to $t$.

\begin{theorem}\cite[Theorem 5]{Doran}\label{T:DSymmSquare} The Picard-Fuchs equation of a family of rank-$19$ lattice-polarized K3 surfaces is a third-order ordinary differential equation which can be written as the symmetric square of a second-order homogeneous linear Fuchsian differential equation. 
\end{theorem}

Recall that a member of $\mathcal{F}_\diamond$ is described by the polynomial
\begin{equation} p = ( c_Q \sum_{x \in \mathscr{Q}} \prod_{k=1}^{18} z_k^{\langle v_k, x \rangle + 1}) + \prod_{k=1}^{18} z_k,
\end{equation}
where $Q$ is the orbit consisting of the nonzero lattice points of $\diamond^\circ$.  To simplify our computations, we set $t = \frac{1}{c_Q}$ and work with hypersurfaces $X_t \in \mathcal{F}_\diamond$ described by the polynomial
\begin{equation}\label{E:S4pencil}  
f = ( \sum_{x \in \mathscr{Q}} \prod_{k=1}^{18} z_k^{\langle v_k, x \rangle + 1}) + t \prod_{k=1}^{18} z_k.
\end{equation}

\begin{theorem}The Picard-Fuchs equation for $\mathcal{F}_\diamond$ is 
\begin{equation}\label{E:picardFuchsEquation}
\derive{\omega}{3} + \frac{6(t^2-32)}{t(t^2-64)}\derive{\omega}{2} + \frac{7t^2-64}{t^2(t^2-64)}\frac{\partial \omega}{\partial t} + \frac{1}{t(t^2-64)}\omega.
\end{equation}
\end{theorem}

\begin{proof}
We apply the the Griffiths-Dwork technique.  Let $\omega = \int \mathrm{Res}\left(\frac{\Omega}{f}\right)$ be a period of the holomorphic form.  The parameter $t$ only appears in a single term of $f$, so the derivatives of $\omega$ have a particularly nice form:
\begin{equation}
\frac{\partial^j}{\partial t^j} \omega = \int (-1)^j j! (z_1 \dots z_{18})^j \mathrm{Res}\left(\frac{\Omega}{f^{j+1}}\right).
\end{equation}

Using the computer algebra system \cite{Magma}, we find that $(z_1 \dots z_{18})^3 \in J$.  We may now apply Equation~\ref{E:reducePoleOrderToric} to compare $\frac{\partial^3}{\partial t^3} \omega$ to lower-order terms.  We conclude that $\omega$ must satisfy Equation~\ref{E:picardFuchsEquation}.
\end{proof}

\begin{corollary}
A general member of $\mathcal{F}_\diamond$ has Picard rank 19.
\end{corollary}

\begin{proof}
By Example~\ref{E:OctahedronAndSkewCube}, a general member of $\mathcal{F}_\diamond$ has Picard rank at least 19.  Families of K3 surfaces of Picard rank $20$ are isotrivial, so if all members of $\mathcal{F}_\diamond$ had Picard rank $20$, $\omega$ would be constant.  But a constant, non-trivial holomorphic two-form $\omega$ cannot satisfy Equation~\ref{E:picardFuchsEquation}.
\end{proof}

We now show that Equation~\ref{E:picardFuchsEquation} is the symmetric square of a second-order differential equation, as predicted by Theorem~\ref{T:DSymmSquare}.  Multiplying Equation~\ref{E:picardFuchsEquation} by $t^2(t^2-64)$ and simplifying, we find that $\omega$ satisfies
\begin{align*}
t^2(t^2-64)\derive{\omega}{3} + 6t(t^2-32)(t^2-64)\derive{\omega}{2} + (7t^2-64)(t^2-64)\frac{\partial \omega}{\partial t} + t(t^2-64)\omega = 0
\end{align*}
Comparing with Equation~\ref{E:symmetricSquare}, we see that the parameters $a_2$, $a_1$, and $a_0$ are given by $a_2 = t(t^2-64), a_1 = 2t^2-64$ and $a_0 = \frac t 4$. Therefore, the symmetric square root of \ref{E:picardFuchsEquation} is
\begin{equation}\label{E:squareRoot}
\derive{\omega}{2} + \frac{(2t^2-64)}{t(t^2-64)}\frac{\partial \omega}{\partial t} + \frac 1 {4(t^2-64)}\omega = 0.
\end{equation}
The symmetric square root is linear and Fuchsian, as expected.

\section{Modularity and Its Geometric Meaning}
All three $ S_4 $ symmetric families of K3 surfaces exhibit ``Mirror Moonshine" (\cite{LY}): the mirror map is related to a hauptmodul for a genus $0$ modular group $\Gamma \subset PSL_2(\bb R)$, which gives a natural identification of the base minus the discriminant locus with $ \bb{H}/\Gamma$ or a finite cover of $ \bb{H}/\Gamma$, where $ PSL_2(\bb R) $ acts on the upper half-plane $ \bb H $ as linear fractional transformations.  Under this identification, the holomorphic solution to the Picard-Fuchs equation becomes a $ \Gamma $-modular form of weight 2.  

In the cases studied in this article, this modularity is not an accident, but rather is a consequence of special geometric properties of the K3 surfaces.

\subsection{Elliptic Fibrations on K3 Surfaces}
We can determine the geometric structures related to modularity by identifying elliptic fibrations with section on these K3 surfaces.  We briefly recall a few facts about elliptic fibrations with section on K3 surfaces.

\begin{definition} An {\em elliptic K3 surface with section} is a triple $ (X, \pi, \sigma) $ where $X$ is a K3 surface, and $ \pi: X \to \bb{P}^1 $ and $ \sigma: \bb{P}^1 \to X $ are morphisms with the generic fiber of $ \pi $ an elliptic curve and $ \pi \circ \sigma = \mathrm{id}_{\bb P^1} $. \end{definition}

Any elliptic curve over the complex numbers can be realized as a smooth cubic curve in $ \bb{P}^2 $ in {\em Weierstrass normal form} 
\begin{equation} \label{Weierstrass} y^2 z = 4x^3 -g_2 x z^2 - g_3 z^3 \end{equation}
Conversely, the equation (\ref{Weierstrass}) defines a smooth elliptic curve provided $ \Delta = g_2^3 -27g_3^2 \neq 0 $.

Similarly, an elliptic K3 surface with section can be embedded into the $ \bb{P}^2 $ bundle $ \bb{P}(\scr O_{\bb P^1} \oplus \scr O_{\bb P^1}(4) \oplus \scr O_{\bb P^1}(6)) $ as a subvariety defined by (\ref{Weierstrass}), where now $ g_2, g_3 $ are global sections of $ \scr O_{\bb{P}^1}(8) $, $ \scr O_{\bb P^1}(12) $ respectively (i.e. they are homogeneous polynomials of degrees 8 and 12).  The singular fibers of $ \pi $ are the roots of the degree 24 homogeneous polynomial $ \Delta = g_2^3 -27g_3^2 \in H^0(\scr O_{\bb P^1}(24)) $.  Tate's algorithm \cite{Tate} can be used to determine the type of singular fiber over a root $ p$ of $ \Delta $ from the orders of vanishing of $ g_2$, $g_3$ , and $\Delta $ at $p$.

\begin{proposition} \label{fibprop} \cite[Lemma 3.9]{CD1} A general fiber of $ \pi $ and the image of $ \sigma $ span a copy of $H$ in $ \mathrm{Pic}(X) $.  Further, the components of the singular fibers of $ \pi$ that do not intersect $ \sigma $ span a sublattice $ S$ of $ \mathrm{Pic}(X) $ orthogonal to this $H$,  and $ \mathrm{Pic}(X)/(H \oplus S) $ is isomorphic to the Mordell-Weil group $MW(X,\pi)$ of sections of $ \pi $.  \end{proposition}

When K3 surfaces are realized as hypersurfaces in toric varieties, one can construct elliptic fibrations combinatorially from the three-dimensional reflexive polytope $ \diamond $. As before, let $ \Sigma $ be a refinement of the fan over faces of $ \diamond $.  Suppose $ P \subset N $ is a plane such that $ \diamond \cap P $ is a reflexive polygon $ \nabla $, let $ m $ be a normal vector to $P$,  and let $ \Xi $ be the fan over faces of $ \nabla $.  Then $P$ induces a torus-invariant map $ V_\Sigma \to \bb P^1 $  with generic fiber $ V_\Xi $, given in homogeneous coordinates by
\begin{equation}\pi_P:  (z_1, \ldots z_r) \mapsto \left[ \prod_{\langle v_i, m \rangle >0} z_i^{\langle v_i, m \rangle}, \prod_{\langle v_i, m \rangle <0} z_i^{-\langle v_i, m \rangle} \right] \end{equation}
Restricting $ \pi_P $ to an anticanonical K3 surface, we get an elliptic fibration.  If $ \nabla  $ has an edge without interior points, this fibration will have a section as well.  See \cite{Kreuzer} for more details.

\begin{example} \label{fibrations} We can use such an elliptic fibration with section to study the ring structure of the Picard group of a generic member $X$ of the family defined by (\ref{E:S4pencil}).  The map $ \pi: V_\Sigma \to \bb{P}^1 $ defined by the procedure above with $ m = (0,1,0)$, is an elliptic fibration with section $ \sigma: \bb{P}^1 \to X $.  

For this particular $ \pi $, examining the singular fibers gives an embedding of the rank 19 lattice $ H \oplus S = H \oplus D_6 \oplus D_6 \oplus A_3 \oplus A_1 \oplus A_1 $ into $\mathrm{Pic}(X) $.  Because this fibration has more than one section, $ H \oplus S \neq \mathrm{Pic}(X) $.  To determine $ MW(X, \sigma) = NS(X)/(H \oplus S) $, we note that the order of this group must divide 16, the square root of the determinant of the intersection matrix of $ H \oplus S $.  By putting the fibration into the Legendre normal form
\begin{equation} \label{K3legendre} y^2z = x(x+z)(x+\frac{stz}{16(1+t)^2}) \end{equation}
 one can see immediately that there are three two-torsion sections, namely $ [0,1,0] $, $ [0,0,1] $, $[-1,0,1] $, and $ [-\frac{st}{16(1+t)^2},0,1] $.  Applying results of \cite{Hitching} show there are no four- or eight-torsion sections.  Hence $ MW(X, \pi) \simeq \bb Z/2 \times \bb Z/2 $.  While this still doesn't completely determine $\mathrm{Pic}(X)$, we know now that it a rank 19 lattice of signature (1,18) with discriminant $ \pm 16 $ which contains the sublattice $  H \oplus D_6 \oplus D_6 \oplus A_3 \oplus A_1 \oplus A_1 $. \end{example}

\subsection{Kummer and Shioda-Inose Structures Associated to Products of Elliptic Curves}
Let $ E_1 $ , $E_2 $ be elliptic curves, thought of as quotients $ \bb{C}/(\bb Z \oplus \bb Z \tau_1) $, $ \bb{C}/(\bb Z \oplus \bb Z \tau_2) $ for $ \tau_1, \tau_2 \in \bb H $.  The action of $ \{ \pm 1 \} $ on $ A = E_1 \times E_2 $ has sixteen fixed points, leading to sixteen nodes on the quotient $ \overline{A} = A / \{ \pm 1 \} $.  The minimal resolution of $ \overline{A} $ is a K3 surface $ Km(A) $ called the {\em Kummer Surface} of $A$.  The Picard group of $ Km(A) $ contains a lattice $DK $ of rank 18, generated by the sixteen exceptional curves of the resolution, together with the strict transforms of the images of $ E_1 \times \{pt \} $, $ \{pt \} \times E_2 $.  Conversely, any K3 surface $X$ with a primitive embedding $DK \into Pic(X) $ is isomorphic to $ Km(A) $ for some $ A = E_1 \times E_2 $ \cite[Prop 3.21]{CD1}.

A Kummer surface carries a symplectic involution $ \beta $, with the minimal resolution of $ Km(A)/\beta $ again a K3 surface $ SI(A) $, called the {\em Shioda-Inose surface} of $A$.  \cite{CD1} shows that $ X \simeq SI(A) $ for $ A = E_1 \times E_2 $ if and only if the rank 18 lattice $ H \oplus E_8 \oplus E_8 $ embeds primitively into $ Pic(X) $.  Generically, this will be exactly the Picard lattice, and so the transcendental lattice will be $ H \oplus H $.

If $ E_1$, $ E_2 $ are $n$-isogenous, i.e. if there exists a degree $n$ morphism $ E_1 \to E_2 $, then the Picard ranks of $ Km(A) $ and $ SI(A) $ have rank 19, with an extra generator corresponding to the strict transform of the graph of the isogeny.  In this case, the Picard lattice of the Shioda-Inose surface will generically be $ H \oplus E_8 \oplus E_8 \oplus \langle -2n \rangle $, and the transcendental lattice will be $ H \oplus \langle 2n \rangle $.

$ E_1 $, $ E_2 $ are $n$-isogenous if and only if, up to the action of $ PSL_2(\bb Z ) $ on $ \bb H$, $ \tau_2 = \frac{-1}{n \tau_1} $.  (Note then that the relation is symmetric; given an isogeny $ E_1 \to E_2$, there exists a dual isogeny $ E_2 \to E_1$.)   Thus if
\begin{equation} \Gamma_0(n) = \left\{  \left.\left(  \begin{array}{cc} a & b \\ c& d \end{array} \right) \in \mathrm{PSL}_2(\bb{Z})
 \: \right| \: c \cong 0 \: (\mathrm{mod} \: n) \right\}  \end{equation}
then the moduli space of ordered pairs of $n$-isogenous elliptic curves is given by $X_0(n) =  \bb H/ \Gamma_0(n) $.  To form the moduli space of products of $n$-isogenous elliptic curves, we need to quotient also by the involution $ \tau \mapsto \frac{-1}{n \tau} $ on $ X_0(n) $.  We call the function $ w_h: \bb H \to \bb H $ defined by $ w_h(\tau) = \frac{-1}{h \tau} $ an {\em Atkin-Lehner map}.  Note that $ w_h $ can be represented by the matrix $ \left( \begin{array}{cc} 0 & \frac{-1}{\sqrt{h}} \\  \sqrt{h} & 0 \end{array} \right) \in PSL_2(\bb R) $, and also that if $ h | n $, then $w_h $ descends to an involution on  $ X_0(n) $.  We write $ \Gamma_0(n) + h $ for the subgroup of $ PSL_2(\bb R) $ generated by $ \Gamma_0(n) $ and $ w_h $, and $ X_0(n) + h $ for the quotient of $ X_0(n) $ by $w_h$ (or equivalently for $ \bb H / (\Gamma_0(n)+h) $.

Thus, $ X_0(n)+n $ is the moduli space of products of $n$-isogenous elliptic curves, and hence also of the Kummer surfaces and Shioda-Inose surfaces associated to such products.  It is important to note, however, that while the transcendental lattices of $ E_1 \times E_2 $ and $ SI(E_1 \times E_2) $ are isomorphic, the transcendental lattice of $ Km(E_1 \times E_2) $ differs from these by scaling by 2.

\subsection{Modular Groups Associated to our Families of K3 Surfaces}
For Examples 3.5  and 3.7, $ \Gamma $ is  $ \Gamma_0(6)+6$ and $ \Gamma_0(6)+3$ respectively (\cite[Prop. 5.4]{HLOY}, \cite[Thm. 2]{Verrill}).

 In these two cases, explicit calculations of Picard lattices in \cite{PS} and \cite{Verrill} show the K3 surfaces have Shioda-Inose structures associated to product of 6- and 3-isogenous elliptic curves respectively. The transcendental lattices of the generic K3's in these pencils are $ H \oplus  \langle 12 \rangle $ and $ H \oplus \langle 6 \rangle $ respectively.  The role of $ \Gamma_0(6)+6 $ for Example 3.5, then, follows from identifying the base of the family with a compactification of the moduli space $X_0(6)+6 $ of $SI(E_1 \times E_2) $ for $E_1 $, $E_2$ $6$-isogenous.  Similarly, in the case of Example 3.7, $ \Gamma_0(6)+3 \subset \Gamma_0(3)+3$, so this example realizes the base of the family as a covering of the moduli space of $ SI(E_1 \times E_2) $ for $ E_1 $ , $E_2$ 3-isogenous.

Example 3.9 is somewhat different.  In this case, the K3 surfaces are not Shioda-Inose surfaces but Kummer surfaces.  To see this, we will use the elliptic fibration of Ex. \ref{fibrations}.  Elliptic fibrations on $ Km(E_1 \times E_2) $ have been classified by \cite{KS}, where in particular they show that generically $ Km(E_1 \times E_2) $ has a fibration giving lattice $ H \oplus D_6 \oplus D_6 \oplus (A_1)^{\oplus 4} $ and Mordell-Weil group $ \bb Z/(2) \oplus \bb Z / (2) $.  If the two elliptic curves are presented in Legendre normal form
\[ y^2 = x(x-1)(x-\lambda_i) \]
then \cite{KS} gives the  Legendre equation for this fibration as 
\[ Y^2 = X(X- u(u-1)(\lambda_2 u - \lambda_1))(X-u(u-\lambda_1)(\lambda_2 u -1)) \]
(where $u$ is an appropriately chosen parameter on the base of the fibration).

 Comparing with our fibration, we see that our family then sits inside the family of $ Km(E_1 \times E_2) $ as a locus where two of the $ A_1 $ singular fibers collide to give an $ A_3 $ singular fiber.  The only possibilities are for $ \lambda_1 = \lambda_2 $ or $ \lambda_1 = 1/\lambda_2 $.  In either case, $ E_1 $ and $ E_2 $ must be isomorphic.  So our family is the family of K3 surfaces of the form $ Km(E \times E) $.  

To determine for what group $ \Gamma $ this family is modular, we consider the symmetric square root (11) of the Picard-Fuchs equation.  By scaling the solutions appropriately, we may put this equation into a projective normal form $ \frac{d^2 f}{dt^2} + Q(t) f= 0$, where
\begin{equation} Q(t) = \frac{\left(t^2-8 t+64\right) \left(t^2+8 t+64\right)}{4 (t-8)^2 t^2 (t+8)^2} \end{equation}
Changing variables via $ t = \frac{1}{i z} $ and comparing with the table of \cite{LW}, we see that
\begin{equation} \Gamma = \Gamma_0(4|2) = \left\{ \left. \left( \begin{array}{cc} a & b/2 \\ 4c & d \end{array} \right) \in PSL_2(\bb{R}) \: \right| \: a,b,c,d \in \bb Z \right\} \end{equation}

\end{document}